\documentclass[letterpaper, 10 pt, conference]{ieeeconf}  

\IEEEoverridecommandlockouts                              
\usepackage{lipsum}
\usepackage{diagbox}
\usepackage{booktabs} 
\usepackage{amsfonts}
\usepackage{mathrsfs} 
\usepackage{amsmath} 
\usepackage{psfrag} 
\usepackage{multirow} 
\usepackage{arydshln}
\usepackage{enumerate}
\usepackage{arydshln}
\usepackage{bbm}
\usepackage{graphicx}
\usepackage{algorithm}
\usepackage[noend]{algorithmic}
\usepackage{color}
\usepackage{mathtools}
\usepackage{subfloat}
\def\mathclap#1{\text{\hbox to 0pt{\hss$\mathsurround=0pt#1$\hss}}}

\newtheorem{definition}{Definition}
\newtheorem{remark}{Remark}
\newtheorem{corollary}{Corollary}
\newtheorem{lemma}{Lemma}
\newtheorem{theorem}{Theorem}

\newcommand{\Real}{\mathbb{R}}
\newcommand{\T}{\mathsf{T}}
\newcommand{\co}{\mathsf c}

\newcommand{\base}{{\bf{e}}}
\newcommand{\eye}{\mathbf{I}}

\newcommand{\clos}{{\text{clos}}}
\newcommand{\inte}{{\text{int}}}

\newcommand{\alphavec}{\boldsymbol{\alpha}}

\newcommand{\Graph}{\mathcal{G}}
\newcommand{\Vertex}{\mathcal{V}}
\newcommand{\VertexG}{\mathcal{V}_{\rm{G}}}
\newcommand{\VertexL}{\mathcal{V}_{\rm{L}}}
\newcommand{\Edge}{\mathcal{E}}

\newcommand{\nG}{N_{\rm{G}}}
\newcommand{\nL}{N_{\rm{L}}}
\newcommand{\nGL}{N}

\newcommand{\CI}{\mathbf{I}}
\newcommand{\CL}{\mathbf{L}}
\newcommand{\CM}{\mathbf{C}}
\newcommand{\Z}{\mathbf{Z}}
\newcommand{\B}{\mathbf{B}}
\newcommand{\x}{\mathbf{x}}

\newcommand{\sgen}{\mathbf{s}^g}
\newcommand{\sload}{\mathbf{s}^l}
\newcommand{\pf}{\mathbf{p}}

\newcommand{\genllim}{\underline{\mathbf{s}}^g}
\newcommand{\genulim}{\overline{\mathbf{s}}^g}

\newcommand{\ang}{\boldsymbol{\theta}}
\newcommand{\f}{{\bf f}}
\newcommand{\h}{{\bf h}}
\newcommand{\M}{{\bf M}}

\newcommand{\pflow}{{\bf p}}
\newcommand{\taueq}{\boldsymbol{\tau}}
\newcommand{\mup}{\boldsymbol{\mu}_+}
\newcommand{\mum}{\boldsymbol{\mu}_-}

\newcommand{\lambdap}{\boldsymbol{\lambda}_+}
\newcommand{\lambdam}{\boldsymbol{\lambda}_-}

\newcommand{\setf}{\Omega_{\f}}
\newcommand{\setsl}{\Omega_{\sload}}

\newcommand{\setpara}{\Omega_{\para}}
\newcommand{\para}{{\boldsymbol{\xi}}}
\newcommand{\setslr}{\widetilde{\Omega}_{\sload}}

\newcommand{\setS}{\mathcal{S}}

\newcommand{\OPF}{\mathcal{OPF}}

\newcommand{\JM}{\mathbf{J}}

\newcommand{\setSgen}{\setS_{\rm G}}
\newcommand{\setSbra}{\setS_{\rm B}}

\newcommand{\Aeq}{{\bf A}_{\rm{eq}}}
\newcommand{\Ain}{{\bf A}_{\rm{in}}}
\newcommand{\beq}{{\bf b}_{\rm{eq}}}
\newcommand{\bin}{{\bf b}_{\rm{in}}}

\makeatletter
\newcommand\fs@spaceruled{\def\@fs@cfont{\bfseries}\let\@fs@capt\floatc@ruled
  \def\@fs@pre{\vspace{0.2cm}\hrule height.8pt depth0pt \kern2pt}%
  \def\@fs@post{\kern2pt\hrule\relax}%
  \def\@fs@mid{\kern2pt\hrule\kern2pt}%
  \let\@fs@iftopcapt\iftrue}
\makeatother

\title{\LARGE \bf
Worst-Case Sensitivity of DC Optimal Power Flow Problems}

\author{James Anderson, Fengyu Zhou,  and Steven H. Low
\thanks{This work is funded by NSF grants CCF 1637598, ECCS 1619352, CNS 1545096, ARPA-E through grant DE-AR0000699 and the GRID DATA program, and DTRA through grant HDTRA 1-15-1-0003.}
\thanks{James Anderson is with the Department of Electrical Engineering and the Data Science Institute at Columbia University, New York, NY. Email: {\tt\small james.anderson@columbia.edu}}%
\thanks{Fengyu Zhou is with the Department of Electrical Engineering, California Institute of Technology, Pasadena, CA, 91125. Email: {\tt\small f.zhou@caltech.edu}}%
\thanks{Steven H. Low is with Department of Electrical Engineering and the Department of Computing and Mathematical Sciences, California Institute of Technology, Pasadena, CA, 91125. Email: {\tt\small slow@caltech.edu}}%
}

\begin{document}

\maketitle
\thispagestyle{empty}
\pagestyle{empty}

\begin{abstract}
In this paper we consider the problem of analyzing the  effect a change in the load vector can have on the optimal power generation in a DC power flow model. The methodology is based upon the recently introduced concept of the $\OPF$ operator. It is shown that for general network topologies computing the worst-case sensitivities is computationally intractable. However, we show that certain problems involving the $\OPF$ operator can be equivalently converted to a graphical discrete optimization problem. Using the discrete formulation, we provide a decomposition algorithm that reduces the computational cost of computing the worst-case sensitivity. A 27-bus numerical example is used to illustrate our results.

\end{abstract}


\section{Introduction}
An optimal power flow (OPF) problem is a mathematical program that searches for the optimal operating point of an electrical power network, subject to power flow equations and operational constraints~\cite{Hun1991survey, Woo2012power, FraR16}. An OPF solution can be used in many contexts; two prominent examples at opposite ends of the time-scale spectrum include market clearing and transmission grid maintenance and expansion.

The non-convex nature of OPF problems with AC power flow constraints has given rise to a large body of work that formulates tractable relaxations. We refer the reader to the following tutorials and survey papers (and the references therein) in order to see the current state of the art~\cite{Low2014convex,Low2014convex2,MolH19}. Despite the success of these approaches, they still lack the maturity and scalability for near real-time use. As a consequence, the DC power flow equations are often the model of choice~\cite{stott2009dc}. They offer the advantage of admitting a linear programming formulation as opposed to a non-convex quadratic program, or in the relaxed case a semidefinite or second-order cone program.  Work in ~\cite{PurMVB05,Li2007dcopf,KriM16} explores how  how good an approximation the DC power flow provides.

In this paper, we work with the more tractable DC optimal power flow problem. We are concerned with determining how the optimal power generation varies as demand and other network parameters change. For example if the cost of generation at a given bus changes, how will this affect the optimal solution? If the demand fluctuates, what happens to the the optimal solution? As the penetration of distributed energy resources increases (for example, causing generation limits to fluctuate depending on the weather), it will become more important to understand such questions.  Similar questions have  been addressed recently in~\cite{Roa19,BieS19,Gri1990optimal}. 

Our problem formulation is as follows; consider a vector of loads which we denote $\sload$ and a vector of power generations $\sgen$. Consider the linear program
\begin{eqnarray}
\underset{\sgen}{\text{minimize}}~~  && \f^{\T}\sgen \nonumber \\
\text{ subject to }& & \Aeq\sgen = \beq(\sload,\textbf b) \label{eq:opf0.b}\\
& &  \Ain\sgen \leq \bin \nonumber
\end{eqnarray}
\label{eq:opf0}
where $\f$ is a vector of generation costs (per unit time). The function $\beq$ is linear in both $\sload$ and $\textbf b$.\footnote{The vector $\textbf b$ is reserved as a placeholder for any constants which may affect the feasible domain where $\sgen$ resides, e.g., non-controllable generation.}
We are concerned with how an optimal $(\sgen)^{\star}$ changes as a function of $\sload$. We refer to such a quantity as the sensitivity (to be made more precise in Section~\ref{sec:sensitivity}), and of particular interest to us is the worst-case sensitivity.

To address this problem, we treat the OPF problem, which is a specific instance of~\eqref{eq:opf0.b} (defined in Section~\ref{sec:Model}), as an operator that maps $\sload$ to $(\sgen)^{\star}$. We refer to this as the $\OPF$ operator which we introduced in~\cite{ZhoAL19}. One of the goals of this paper is to introduce the reader to this new framework. In Section~\ref{sec:OPFop} we introduce the operator, and provide conditions under which its derivative (i.e., sensitivity) is well defined. We further show that this seemingly continuous property can be equivalently formulated as a discrete optimization problem. In Section~\ref{sec:sensitivity} a formal definition of the \emph{worst-case} sensitivity is given and three sample problems are described. Section~\ref{sec:alg} highlights some new insights that can be gained by viewing an OPF sensitivity problem in the discrete setting, and finally in Section~\ref{sec:example} we provide two numerical examples.

\subsection*{Notation}
Bold letters such as $\mathbf z$ and $\mathbf Y$ denote vectors and matrices, lower-case letters are reserved for vectors, and upper-case for matrices. When it is not clear from context, dimensions of specific matrices are indicated with a superscript,  for example the vector of $n$ zeros is $\mathbf 0^n$ and the $m\times m$ identity matrix is $\mathbf{I}^m$ (non-square matrix  definitions follow in an obvious manner). The only exceptions to this rule are $\sgen$ and $\sload$ which are reserved for generation and load vectors (as defined in Section~\ref{sec:Model}). We will frequently make use of matrices which are comprised of a subset of rows of another matrix; Let $\mathcal{I}$ denote a set of positive integers, then $\mathbf Y_{\mathcal I}$ is the matrix formed by stacking the rows of $\mathbf Y$ indexed by $\mathcal I$ on top of each other. All index sets are assumed to be ordered sets. Calligraphic letters are reserved for sets. An inequality constraint is said to be {\it binding} at the optimum if equality is achieved.

\section{Power Flow Model and Sensitivity} 
In Section~\ref{sec:Model} we describe the network and  power flow model throughout this work. In Sections~\ref{sec:sets} we introduce some parameter sets that we will make use of in the sequel.
\subsection{Network Model and DC Power Flow}\label{sec:Model}
We consider a DC power flow formulation. The network is modeled by an undirected graph $\Graph(\Vertex,\Edge)$, where the edge-set $\Edge \subseteq \Vertex \times \Vertex$ indicates there is an edge between two vertices. The set of vertices can be further classified into generator and load buses such that $\Vertex = \VertexG \cup \VertexL$ with $\VertexG$ denoting generator buses and $\VertexL$ denoting load buses. We assume that there are $N$ vertices in the network with $\nG$ generator buses and $\nL$ load buses. The cardinality of the edge-set is $E$. For simplicity, we assume that no vertex is both a generator and a load bus. Let $\CM \in \mathbb{R}^{N\times E}$ denote the incidence matrix of $\Graph$. Finally, let $\B=\text{diag}(b_1,\hdots, b_E)$ with $b_e>0$ indicating the susceptance of branch $e$. The Laplacian matrix is defined as $\CL= \CM\B\CM^{\T}$.

We denote the load and generation vectors by $\sgen \in \mathbb{R}^{\nG}$ and $\sload \in \mathbb{R}^{\nL}$ respectively. We index the vertices as follows, $\Vertex = \{v_1, \hdots, v_{\nG},v_{\nG+1}, \hdots,v_{\nG+\nL} \}$.  Thus, $\sgen_i$ refers to the generation on bus $i$, while $\sload_i$ refers to the load on bus $\nG + i$. We will refer to bus $\nG + i$ as load $i$ for simplicity. The power flow on edge $e \in \Edge$ is denoted as $\pf_e$, and $\pflow = [ p_1,\hdots,p_E]^\T$ is the vector of all branch power flows.

The DC power flow model assumes that the voltage magnitudes are fixed and known, and the lines are lossless.  The DC-OPF problem is a linear program:
\begin{subequations}
\begin{eqnarray}
\underset{\sgen, \ang}{\text{minimize}} ~~ && \f^{\T}\sgen
\label{eq:opf1.a}\\
\text{ subject to }& & \ang_1 = 0
\label{eq:opf1.b}\\
& &  \CL \ang = 
\left[ \begin{array}{c}
\sgen \\
-\sload
\end{array} \right] 
\label{eq:opf1.c}\\
& & \genllim \leq\sgen\leq \genulim
\label{eq:opf1.d}\\
& &  \underline{\pflow}\leq\B\CM^{\T}\ang\leq\overline{\pflow}.
\label{eq:opf1.e}
\end{eqnarray}
\label{eq:opf1}
\end{subequations}
The decision variables are the power generations $\sgen$ and voltage angles $\ang\in\Real^{\nGL}$. The cost vector $\f\in\Real_+^{\nG}$ is the unit cost for each generator and constraint~\eqref{eq:opf1.b} indicates that bus $1$ has been set as the slack-bus. All voltage magnitudes are fixed at $1$. In constraint \eqref{eq:opf1.c}, we let the injections for generators be positive while the injections for loads be $-\sload$. The upper and lower limits on the generations are set as $\genulim$ and $\genllim$, respectively, and $\overline{\pflow}$ and $\underline{\pflow}$ are the limits on branch power flows. We assume that~\eqref{eq:opf1} has a non-empty feasible set. 
\subsection{Set Definitions}\label{sec:sets}
The main objective of this paper is to quantify how changes in the load vector $\sload$ impact the optimal power generations as computed by solving~\eqref{eq:opf1}. Specifically we would like to obtain estimates of the derivatives $\frac{\partial (\sgen_i)^{\star}}{\partial \sload_j} $ for all $i\in [1,\hdots, \nG]$, $j\in [1,\hdots, \nL]$, where $(\sgen)^{\star}$ is the optimal vector of generations returned by~\eqref{eq:opf1}.\footnote{For the remainder of the paper, we shall adopt the notation $[\nL]$ to denote $1,\hdots, \nL$.}

As mentioned in the introduction, the sensitivity of an optimal power flow problem is a useful quantity to have in many practical applications. However, obtaining expressions for these derivatives, and as we shall see later, obtaining the ``worst-case" sensitivity is both analytically and computationally challenging. In previous work \cite{ZhoAL19} we rigorously formulated the sensitivity problem, characterized problem instances for which the binding constraints would not change under perturbation, and ensured that the optimal solution is unique and derivatives exist. In order to make the previous definitions precise, and before we can introduce the $\OPF$ operator, we require the following definitions.

Let $\para$ be a vector of $2\nG+2E$ network limits arranged as
\begin{equation*}
\para : = [(\genulim)^\T, (\genllim)^\T,\overline{\pflow}^\T,\underline{\pflow}^\T ]^\T.
\end{equation*}
Define the sets
\begin{align*}
&\setpara:=\{\para|\genllim\geq 0, \eqref{eq:opf1.b}-\eqref{eq:opf1.e}~\text{are feasible for some}~\sload>0\},\\
&\setsl(\para):=\{\sload | \sload>0, \eqref{eq:opf1.b}-\eqref{eq:opf1.e}~\text{are feasible}\}~\text{for $\para\in\setpara$}.
\end{align*}
The set $\setsl(\para)$ is  convex and non-empty. When $\para$ is clear we will simply refer to this set as $\setsl$. These two sets collect the parameters we care about. The next set ensures~\eqref{eq:opf1} has a unique solution:
\begin{align*}
\setf &:=\{\f\geq0~|~\forall\para\in\setpara,\sload\in\setsl(\para),~\text{\eqref{eq:opf1} has a  unique }\\
&\hspace{2cm} \text{solution, and $\geq\nG-1$ nonzero dual} \\ &\hspace{2cm} \text{variables at the optimal point.\} }\
\end{align*}
The definition above is in fact more restrictive than what is needed for uniqueness.
We impose the additional constraint on the number of non-zero dual variables as it paves the way for further desirable properties, 
where we show, that up to perturbation all the binding constraints are independent and there are exactly $\nG-1$ of them. 

 For completeness, the dual to the OPF-problem~\eqref{eq:opf1} is provided in the appendix. With these definitions in hand, we are ready to define the $\OPF$ operator abstraction of~\eqref{eq:opf1}.
\section{The $\OPF$ Operator and Worst-Case Sensitivity}\label{sec:OPFop}
\subsection{Existence and Smoothness}
Instead of dealing with the convex program~\eqref{eq:opf1} directly, we instead treat it as an operator that maps loads to optimal generations.
\begin{definition}
Assume $\f \in \setf$. Let $\OPF$ be the operator $\OPF:\setsl \rightarrow \mathbb{R}^{\nG}$ such that $\OPF(\sload)$ returns an optimal solution to~\eqref{eq:opf1}, i.e. $(\sgen)^{\star}=\OPF(\sload)$. 
\end{definition}
We collect a few observations pertaining to $\OPF:$
\begin{itemize}
\item The assumption that $\f \in \setf$ ensures that $\OPF$ is a singleton, i.e. it returns a unique element.
\item $\OPF$ defines a parametric linear program. Solution sets to parametric LPs are both upper and lower hemi-continuous, thus the $\OPF$ solution set inherits hemi-continuity. Furthermore, when $\f\in\setf$,  $\OPF$ is  continuous.
\item $\setf$ is dense in $\mathbb{R}_+^{\nG}$ (See Proposition 1 in \cite{ZhoAL19}). Thus, if $\f\notin \setf$ applying a small perturbation to $\f$ will with probability 1 ensure that $\f' \in \setf$.
So the assumption that $\f\in\setf$ is mild.
\end{itemize}
We require one final set definition before we can state the differentiability properties of the $\OPF$ operator. 
\begin{align*}
\setslr(\para,\f):=&\{\sload\in\setsl(\para)~|~\eqref{eq:opf1}~\text{has exactly }  \nG-1 \\ & \hspace{2.6cm}\text{ binding inequalities.}\}
\end{align*}
The $\nG-1$ binding inequalities condition above ensures (when combined with the restriction of $\f$ to $\setf$) that the set of binding inequalities are independent. This technical assumption is required in the proof of Theorem~\ref{thm:deriv}.
\begin{theorem}\label{thm:deriv}
Assume that $\f \in \setf$. Then there exists a dense set $\widetilde\Omega_{\para}(\f) \subseteq \setpara$ such that  for all $ \para \in \widetilde\Omega_{\para}(\f)$ the following hold:
\begin{enumerate}
\item $\clos(\inte(\setsl(\para)))=\clos(\setsl(\para))$ 
\item $\setslr(\para,\f)$ is dense in $\setsl(\para)$. 
\end{enumerate}
Then, when $\f \in \setf$ and $\para \in \widetilde\Omega_{\para}(\f)$, the derivative $\partial_{\sload} \OPF(\sload)$ exists for $\sload\in\setslr$, and the set of binding constraints remain unchanged in some neighborhood of $\sload$.
\end{theorem}
\begin{proof}
The proof of the two topological properties of $\setslr(\para,\f)$ is somewhat involved but can be found in Appendix C of~\cite{ZhoAL19}. With these definitions in hand, by construction, the appropriate sets possess the  necessary topological properties such that when combined with the OPF problem~\eqref{eq:opf1}, they satisfy all the necessary conditions in Lemma 4.1 of \cite{Gri1990optimal} which guarantee that the derivatives always exist and binding constraints do not change locally. 

\end{proof}
\begin{corollary}\label{Co:independent}
If  $\sload\in\setslr$, the $\nG-1$ binding inequalities, along with $N+1$ equality constraints, are independent. 
\end{corollary}

Now, suppose at point $\sload$, the set of generators corresponding to binding inequalities is $\setSgen\subseteq\VertexG$, while the set of branches corresponding to binding inequalities is $\setSbra\subseteq\Edge$. As a consequence of 1) and 2) in Theorem~\ref{thm:deriv} we obtain the following:
\begin{corollary}
When $\f\in\setf, \para\in\widetilde\Omega_{\para}(\f), \sload\in\setslr(\para,\f)$, we have 
\begin{equation*}
|\setSgen|+|\setSbra|=\nG-1.
\end{equation*}
\end{corollary}

We have placed a lot of emphasis on sets being dense, specifically, $(\setf, \setslr, \widetilde\Omega_{\para})$ being dense with respect to $( \mathbb{R}^{\nG}_+, \Omega_{\sload},\Omega_{\para})$. The reason for this is that if the parameter of the OPF problem under consideration does not satisfy the necessary assumptions, then there exists another parameter arbitrarily near by, that does. Thus applying a perturbation to the parameter will provide an OPF problem that does satisfy the necessary conditions for the derivative to be well defined. In summary when $\f,\para,\sload$ belong to $(\setf, \setslr, \widetilde\Omega_{\para})$, we have shown that $\OPF$ is well defined, has a unique solution, is differentiable, and at the optimal solution the binding constraints are independent.

\subsection{The Jacobian}\label{sec:jac}
The Jacobian matrix is one of the most fundamental components in sensitivity analysis. In this section, we will derive a closed-form expression for the Jacobian matrix that links the structure (properties of $\Graph$) with features of the solution to the OPF problem~\eqref{eq:opf1}. To ensure that the appropriate partial derivatives exist (almost everywhere), we assume w.l.o.g. that  $\f\in\setf, ~\para\in\widetilde\Omega_{\para}(\f), ~\text{and }\sload\in\setslr(\para,\f)$. We define the Jacobian matrix element-wise as
\begin{equation}\label{eq:jac_opf}
\left[ \JM(\sload;\f,\para)\right]_{i,j} := \frac{\partial (\sgen_i)^{\star}}{\partial \sload_j} = \frac{\partial [\OPF(\sload)]_i}{\partial \sload_j}.
\end{equation}
We will often omit the parameters $\para$ and $\f$ and the argument $\sload$ to lighten the notation. It should be clear though, that the Jacobian is defined with respect to a specific parameter realization $(\f,\para)$ and is evaluated at a given $\sload$.

The next lemma shows that $\JM(\sload;\f,\para)$ could be re-parameterized as a function of $\setSgen$ and $\setSbra$.

\begin{lemma}
The Jacobian $\JM(\sload;\f,\para)$ written as a function of the binding constraint takes the form
\begin{equation*}
\JM(\setSgen,\setSbra) = -\mathbf{\Psi}(\eye^N_{[N_L]})^\T,
\end{equation*}
where $\mathbf{\Psi} = \CI^N_{\VertexG}\CL \Z(\setSgen,\setSbra)^\T$ and
\begin{equation*}
\Z(\setSgen,\setSbra)^\T = \left[\begin{array}{c} \CI^N_{\VertexL}\CL\\ \CI^N_{\setSgen} \CL \\ \CI^E_{\setSbra}\B\CM^\T \\ \base_1^\T \end{array} \right]^{-1}.
\end{equation*}
\end{lemma}
\begin{proof}[sketch]
We first write down a mapping $\Delta$ from decision-variable space to parameter-- and load--space :
\begin{equation*}
 \left[\begin{array}{c} \mathbf{0} \\ \hdashline[0.4pt/1pt]  \sload \\ \mathbf \Gamma^\T \para \\ 0 \end{array}\right] = \Delta \left[\begin{array}{c}\sgen \\ \hdashline[0.4pt/1pt] \ang \end{array}\right]
\end{equation*}
where 
\begin{equation*}
\Delta = \left[ \begin{array}{c; {0.4pt/1pt} c} \CI^{N_G}  & \CI^N_{\VertexG}\CL \\ \hdashline[0.4pt/1pt]  \mathbf{0}^{N_L\times N_G} & \CI^N_{\VertexL} \CL\\ \CI^{N_G}_{\setSgen} & \mathbf{0}^{|\setSgen |\times N} \\  \mathbf{0}^{|\setSbra |\times N_G} & \CI^E_{\setSbra}\B \CM^\T \\ \mathbf{0}^{1 \times N_G} & \base_1^\T
\end{array}\right],
\end{equation*}
and $\mathbf{\Gamma}$ is a matrix with columns given by standard basis vectors such that $\mathbf \Gamma^T \para$ is a vector whose elements are equal to the capacity and branch flow limits of the binding constraints. It can be shown (proof omitted) that the rows of $\Delta$ are independent and thus $\Delta^{-1}$ exists. This inverse provides a map in the reverse direction (from load- and parameter-space to decision variable space.) As we are only concerned with the map to the optimal generations we only want an expression for the $(1,2)$-block of $\Delta^{-1}$, which we denote as $\mathbf{\Psi}$, where the partitions are as indicated by the dashed lines. The expression follows by applying the block-matrix inversion lemma and some routine linear algebra.
\end{proof}

We end this section with a result that links the $\OPF$ formulation of the Jacobian~\eqref{eq:jac_opf}, with the binding constraint formulation proved above. Before doing so we define the notion of  independent binding constraints. Consider the linear program $\{\text{minimize}_{\x}~ \mathbf c^\T\x \text{ s.t. }\mathbf{A}\x \le \mathbf b\}$, a solution to this LP will have several binding constraints, i.e., $(\mathbf{A}\x)_i = \mathbf b_i$ for $i$ in some set $\mathcal{S}$. We say that the binding constraints are independent if the rows $\mathbf A_{\mathcal S}$ are linearly independent.

\begin{theorem}\label{thm:Jac}
The Jacobian $\JM(\sload;\f,\para)$ derived from $\OPF$, and the Jacobian $\JM(\setSgen, \setSbra)$ expressed as a function of the binding constraints satisfy 
\begin{align*}
\mathbf{range}(\JM(\sload;\f,\para)) = \mathbf{range}(\JM(\setSgen, \setSbra))
\end{align*}
when
\begin{align*}
 \f\in\setf, \quad \para\in\widetilde\Omega_{\para}(\f), \quad \sload\in\setslr(\para,\f),
\end{align*}
and
\begin{align}\label{eq:disc_cons}
\setSgen \subseteq \VertexG,~ \setSbra \subseteq \Edge,~ |\setSgen |+|\setSbra |= N_G-1, ~ \setSgen \perp \setSbra , 
\end{align}
and $(\setSgen, \setSbra)$ correspond to the OPF problem defined by $\sload,\f,\para$. 
\footnote{The {\it range} refers to the set of values that $\JM(\sload; \f, \para)$ or $\JM(\setSgen, \setSbra)$ could take, rather than the column space of $\JM(\sload; \f, \para)$ or $\JM(\setSgen, \setSbra)$.}
\end{theorem}
\begin{proof}
The proof of this result can be found in \cite[\S 4.3]{ZhoAL19}. 
\end{proof}

The notation $\setSgen \perp \setSbra$ in~\eqref{eq:disc_cons} is used to indicate that $\mathbf A_{\setSgen \cup \mathcal{S}_{\mathrm B}}$  has linearly independent rows, where $\mathbf A$ is the standard form version of~\eqref{eq:opf1} -- this is explicitly formulated in Appendix~\ref{sec:standard}. The set $\mathcal S_{\mathrm{B}}$ is an index set that corresponds to rows of the constraint matrix $\mathbf A$ (see appendix ~\ref{sec:standard} ) which in turn  corresponds to binding constraints on the edges as identified in $\setSbra$. 

What we find interesting (and hopefully useful) about Theorem~\ref{thm:Jac} is that the two formulations capture completely different aspects of the problem, yet they are equivalent.\footnote{Note that they are only equivalent when $(\setSgen,\setSbra)$ and $\f,\para,\sload$ map to each other through the same OPF.} $\JM(\sload;\f,\para)$ depends on a continuous optimization problem~\eqref{eq:opf1} and involves physical parameters such as loads, limits, and a cost function. 
In contrast, $\JM(\setSgen,\setSbra)$ has the discrete input space and depends purely on $\Graph (\Vertex, \Edge)$ given the binding sets $(\setSgen,\setSbra)$.

\begin{remark}

It is worth noting that though one direction of Theorem \ref{thm:Jac} that 
$\mathbf{range}(\JM(\sload;\f,\para)) \subseteq \mathbf{range}(\JM(\setSgen, \setSbra))$ is quite straightforward,
the other direction is in fact not obvious at the first glance as it requires any legal choice of binding sets $\setSgen$ and $\setSbra$ be exactly achieved in at least one realization of OPF.\footnote{More specifically, such realization needs to satisfy the condition that $\f\in\setf, \para\in\widetilde\Omega_{\para}(\f), \sload\in\setslr(\para,\f)$.}
As we will see in the following subsection, a direct consequence is while deriving the worst-case sensitivity, it is tight to change the decision variables from $(\sload;\f,\para)$ to $(\setSgen, \setSbra)$.
\end{remark}

\subsection{Worst Case Sensitivity}\label{sec:sensitivity}
The sensitivity of a solution to an optimal power flow problem has many immediate practical uses, as discussed in the introduction. There are indeed many sensitivity problems that can be formulated. Before we do so, it will be helpful to make concrete the link between continuity and the Jacobian matrix. To avoid notational overload we will refer to arbitrary functions and sets and then provide the definition specific to $\OPF$. 

Recall that a function $\h: \mathcal D\rightarrow \Real^n$ with $\mathcal D$ an open subset of $\Real^n$ is said to be Lipschitz on $\mathcal D$ if there exists some $L\ge 0$ such that
\begin{equation}\label{eq:lip}
\|\h(\x )- \h(\x ') \| \le L \|\x - \x'\|
\end{equation}
for all $\x,\x' \in \mathcal D$. Suppose that the Jacobian $\JM := \frac{\partial \h}{\partial \x}$ exists and is continuous on $D$. 
Then if for some convex subset $\mathcal B \subseteq \mathcal D$, there exists a constant $K \ge 0$ such that 
\begin{equation*}
\left\| \frac{\partial \h(\x)}{\partial \x}\right\| \le K
\end{equation*}
on $\mathcal B$, then~\eqref{eq:lip} holds for all $\x,\x' \in \mathcal B$ with $L=K$. This establishes a clear link between a bound on the norm of the Jacobian and the Lipschitz constant of a function. We now define the notion of Lipschitz continuity for a generator-load pair, and then formulate three sensitivity definitions.

\begin{definition}\label{def:lip}
The matrices $\CM, \B$ are fixed, as is the cost vector $\f \in \setf$ and $\para \in \widetilde\Omega_{\para}(\f)$. 
Select a generator $i$ and load $j$. The pair $(i,j)$ is said to be $C$-Lipschitz if for all $\delta>0$ and $\alphavec, \alphavec' \in \setsl(\para) $ such that $|\alphavec_j - \alphavec_j' |\le \delta$ and $\alphavec_k = \alphavec_k'$ for all $k \neq j$, we have that 
\begin{equation*}
|\OPF_i(\alphavec) - \OPF_i(\alphavec')|< C\delta,
\end{equation*}
where $\OPF_i(\cdot)$ denotes the $i^{\text{th}}$ coordinate of $\OPF(\cdot)$.
\end{definition}
This  Lipschitz-like definition forms the basis of the sensitivity analysis formulation we are proposing. In the remainder of this section we formulate several sensitivity problems that will be of interest to grid operators.
\begin{remark}
Recall that in our notation, when we refer to the $(i,j)$-generator-load pair, this corresponds to vertices $(v_i, v_{N_{\mathrm G}+j})$.
\end{remark}
\subsubsection{\textbf{Problem 1}} \textbf{SISO Sensitivity }\newline 
In this formulation we consider the problem of computing  the worst-case sensitivity of the generator-load pair  $(i,j)$. We use SISO to mean single-input, single-output, i.e. the change in one output when one input is changed. Recall that according to our indexing of vertices, load $j$ corresponds to the vertex $v_{N_G+j}$. 
\begin{definition}
 The (SISO) sensitivity of generator $i$ with respect to load $j$ is the minimum value which we denote by $C_{i\leftarrow j}$, such that $(i,j)$ is a  $C_{i\leftarrow j}$-Lipschitz pair,
 i.e., $C_{i\leftarrow j}$ is the minimal $C$ such that $|\OPF_i(\alphavec)-\OPF_i(\alphavec')|<C\delta$ for every $\alphavec,\alphavec'$ that differ only in their $j$\textsuperscript{th} coordinates with $|\alphavec_j-\alphavec'_j|\leq \delta$.
\end{definition}

\subsubsection{\textbf{Problem 2}} \textbf{Worst-Case SISO Sensitivity}\newline 
In the SISO sensitivity formulation, it was assumed that all the network parameters and the OPF cost function were fixed. In this version of the problem we allow the  network parameters to change (apart from those which define the network structure, e.g., $\CM$, the graph incidence matrix). 
\begin{definition}\label{def:wc}
The worst-case (SISO) sensitivity of generator $i$ with respect to load $j$ is 
\begin{equation}\label{eq:wcsen}
C_{i\leftarrow j}^{\mathrm{wc}}:=\underset{\f \in \setf }{\text{max}}~~\underset{\para \in\widetilde\Omega_{\para}(\f)  }{\text{max}}~C_{i\leftarrow j}.
\end{equation}
\end{definition}
The ability to allow parameter variations means  $C_{i\leftarrow j}^{\mathrm{wc}}$ provides information about various network scenarios. For example,  a generator instantaneously  going offline can be modeled by $\overline{\pflow},\underline{\pflow} \rightarrow \epsilon$, where $\epsilon$ is a small constant. Taking $\epsilon = 0$ would potentially break the independence conditions we require. In practice a small constant such as $\epsilon =  10^{-5}$ suffices. The quantity $C_{i\leftarrow j}^{\mathrm{wc}}$ plays an important role in  releasing power flow data in a differentially private manner~\cite{ZhoAL19a}. 

\subsubsection{\textbf{Problem 3}} \textbf{MISO Sensitivity}
\newline
Consider a set of $m$ load buses $\Vertex_{\mathrm L}'\subseteq \Vertex_\mathrm{L}$ and  let $\mathcal{L}$ denote the set of indices corresponding to those loads. The MISO part of the definition refers to the fact that here, we are interested in how a single output (generation) changes when multiple inputs (loads) are allowed to simultaneously change. To make this definition concrete, we must first  modify Definition~\ref{def:lip}. 
\begin{definition}
Assume that $\CM, \B$ are fixed, as is the cost vector $\f \in \setf$ and $\para \in \widetilde\Omega_{\para}(\f)$. We say that $(i,\mathcal L )$ is $C^{(m)}$-Lipschitz if for all $\delta >0$ and $\alphavec, \alphavec' \in \setsl(\para) $ such that $\|\alphavec - \alphavec' \|\le \delta$ and $\alphavec_k = \alphavec'_k$ for all $k \notin \mathcal L$, there exists a constant $C^{(m)}$ such that
\begin{equation*}
\|\OPF_i(\alphavec) - \OPF_i(\alphavec')\|< C^{(m)}\delta. 
\end{equation*}
\end{definition}
\begin{definition}
The (MISO) sensitivity of generator $i$ with respect to the set $\mathcal L$ of loads, denoted by $C_{i \leftarrow \mathcal L}$, is the minimum value of $C^{(m)}$ such that $(i,\mathcal L)$ is $C^{m}$-Lipschitz. 
\end{definition}
The worst-case MISO sensitivity problem can then be derived analogously to Definition~\ref{def:wc}.

\section{A Structural Result and Algorithm}\label{sec:alg}
In Section~\ref{sec:jac} we provide a derivation of the Jacobian matrix based on the sets of binding constraints $\setSbra$ and $\setSgen$. In this section we show that the combination of these sets and the OPF parameters allows us to gain structural information about the power network. The following lemma holds for any graph topology and provides a taste for the type of results we can aim for. For different graph structures we have similar results - these will be presented in a future paper.
In this section, the term graph partition refers to the process of reducing a graph into several smaller disjoint components by removing a subset of edges. A subset of edges that when removed disconnects the graph, is  known as a cut-set.
\begin{lemma}\label{lm:subgraph_is_not_full}
Suppose $\setSbra$ partitions $\Graph$ into $m$ disjoint subgraphs $\{\Graph_i(\Vertex_i,\Edge_i)\}_{i=1}^m$,
where $\cup_i \Vertex_i=\Vertex$ and $(\cup_i \Edge_i)\cup\setSbra=\Edge$.
Then for any $i$, we have $\setSgen^\co\cap\Vertex_i\neq\emptyset$ where $\setSgen^\co:=\VertexG\setminus\setSgen$.
\end{lemma}
\begin{proof}
If not, then all the generators in $\Vertex_i$ are binding. 
For fixed $i$,
let
\begin{align*}
\mathbf{T}:=
\left[
\begin{array}{c}
\CL\\
\B\CM^\T
\end{array} 
\right]
\end{align*}
and let $\Edge_0$ be the subset of $\setSbra$ satisfying $\forall e=(u,v)\in\Edge_0$, $u\not\in\Vertex_i$ and $v\in\Vertex_i$.
Consider the following constraints:
\begin{subequations}
\begin{eqnarray}
&\mkern-24mu\mathbf{T}_{\{j\}}\cdot\ang=\sgen_j &\text{for all}~j\in\Vertex_i\cap\VertexG
\label{eq:tightcon.a}\\
&\mkern-24mu\mathbf{T}_{\{j\}}\cdot\ang=-\sload_{j-\nG} &\text{for all}~j\in\Vertex_i\setminus\VertexG
\label{eq:tightcon.b}\\
&\mkern-24mu\sgen_j\in\big\{\genulim_j,\genllim_j\big\} &\text{for all}~j\in\Vertex_i\cap\VertexG
\label{eq:tightcon.c}\\
&\mkern-24mu\mathbf{T}_{\{\nGL+e\}}\cdot\ang\in\big\{\overline{\pflow}_e,\underline{\pflow}_e\big\} &\text{for all}~e\in\Edge_0
\label{eq:tightcon.d}
\end{eqnarray}
\label{eq:tightcon}
\end{subequations}
Note that
\begin{align*}
&\sum\limits_{j\in\Vertex_i}\mathbf{T}_{\{j\}}=
\sum\limits_{j\in\Vertex_i}\sum\limits_{e=(j,j')\in\Edge}(b_e\base_{j}^{\T}-b_e\base_{j'}^{\T})\\
=&\sum\limits_{j\in\Vertex_i}\sum\limits_{\tiny\substack{e=(j,j')\\e\in\Edge_i}}(b_e\base_{j}^{\T}-b_e\base_{j'}^{\T})
+\sum\limits_{j\in\Vertex_i}\sum\limits_{\tiny\substack{e=(j,j')\\e\in\Edge_0}}(b_e\base_{j}^{\T}-b_e\base_{j'}^{\T})\\
=&\sum\limits_{\tiny\substack{e=(j,j')\\e\in\Edge_i, j<j'}}(b_e\base_{j}^{\T}-b_e\base_{j'}^{\T})+(b_e\base_{j'}^{\T}-b_e\base_{j}^{\T})\\
&+\sum\limits_{j\in\Vertex_i}\sum\limits_{\tiny\substack{e=(j,j')\\e\in\Edge_0}}(b_e\base_{j}^{\T}-b_e\base_{j'}^{\T})\\
=&\sum\limits_{\tiny\substack{e=(u,v)\in\Edge_0\\u\not\in\Vertex_i, v\in\Vertex_i}}(b_e\base_{v}^{\T}-b_e\base_{u}^{\T})
=\sum\limits_{\tiny\substack{e=(u,v)\in\Edge_0\\u\not\in\Vertex_i, v\in\Vertex_i}}\CM_{v,e}\mathbf{T}_{\{N+e\}}.
\end{align*}
Here $\CM_{v,e}$ is the $(v,e)$ element of matrix $\CM$.
As a result, the summation of  \eqref{eq:tightcon.a} to \eqref{eq:tightcon.c} is linearly dependent to \eqref{eq:tightcon.d}, and contradicts Corollary \ref{Co:independent}.
Thereby, $\setSgen^\co\cap\Vertex_i\neq\emptyset$.
\end{proof}

Lemma~\ref{lm:subgraph_is_not_full} tells us  that under our assumptions, if we view $\setSbra$ as a cut of the graph, i.e., the graph $\Graph(\Vertex, \Edge\setminus \setSbra)$ has a disconnected component, then each subgraph contains at least one generator that is neither at maximum, nor minimum,  power production.  This information is  useful for network planning, for example when deciding where to add extra generation capacity in a network with transmission lines that are often saturated.

\subsection{Computing $C_{i\leftarrow j}^{\mathrm{wc}}$}
 It can be shown that the worst-case sensitivity is computed by solving a discrete optimization problem based on the binding constraint formulation of the Jacobian matrix. The formal statement and its proof are beyond the scope of the current paper, but we include it here for completeness: 
\begin{equation}\label{eq:disc-opt1}
C_{i\leftarrow j}^{\mathrm{wc}}=\max\limits_
{\tiny\substack{\setSgen\in\VertexG,\setSbra\in\Edge \\ |\setSgen|+|\setSbra|=\nG-1\\\setSgen\perp\setSbra}}|\JM_{i,j}|.
\end{equation}
Note that the constraints in the problem above are exactly~\eqref{eq:disc_cons} from Theorem~\ref{thm:Jac}. 
Unfortunately \eqref{eq:disc-opt1} is a non-convex, discrete optimization problem and thus intractable in general. However,  we provide a decomposition algorithm that produces small  sub-graphs such that a brute-force search is possible. Future work will examine how to relax~\eqref{eq:disc-opt1} to a more tractable problem.

\subsection{Decomposition Algorithm}
\floatstyle{spaceruled}
\restylefloat{algorithm}
\begin{algorithm} 
\caption{Decomposition of the computation of $C_{i\leftarrow j}^{\mathrm{wc}}$.}
\label{alg2}
\begin{algorithmic}
\renewcommand{\algorithmicrequire}{\textbf{Input:}}
\renewcommand{\algorithmicensure}{\textbf{Output:}}
\newcommand{\algorithmiccompute}{\textbf{compute }}
\newcommand{\CALL}{\STATE \algorithmiccall }
\newcommand{\algorithmiccall}{\textbf{call }}
\newcommand{\COMPUTE}{\STATE \algorithmiccompute }
    \REQUIRE $\B$, $\CM$, $i\in[\nG]$, $j\in[\nL]$, $\Graph(\VertexG\cup\VertexL,\Edge)$
    \ENSURE $C_{i\leftarrow j}^{\mathrm{wc}}$
    \FOR{$e=(u,v)$ in $\Edge^{\rm bri}$}
    	\IF{$u,v\not\in\VertexG$ and $v_i\Leftrightarrow v_{j+\nG}$ in $\Graph(\Vertex,\Edge\setminus\{e\})$}
		\STATE $e$ partitions $\Graph$ into $\Graph_1$ and $\Graph_2$ (assume $v_i, v_{j+\nG}$ are both in $\Graph_1$)
		\IF{$\Graph_2$ contains any vertex in $\VertexG$}
			\STATE Replace $\Graph_2$ by a single generator
		\ELSE
			\STATE Replace $\Graph_2$ by a single load
		\ENDIF
	\ENDIF
    \ENDFOR
    \STATE Find a shortest path connecting $v_i$ and $v_{j+\nG}$
    \STATE Get $\{\Graph_{l}\}_{l=1}^m$ and add $p_l$, $q_l$ to subgraphs
    \FOR{$l=0$ to $m-1$}
    	\CALL subroutine to compute $C_{p_l\leftarrow q_{l+1}}^{\mathrm{wc}}$
    \ENDFOR
    \STATE $C_{i\leftarrow j}^{\mathrm{wc}}\leftarrow \prod_{l=0}^{m-1} C_{u_l\leftarrow v_{l+1}}^{\mathrm{wc}}$
    \RETURN $C_{i\leftarrow j}^{\mathrm{wc}}$
\end{algorithmic}
\end{algorithm}

In this subsection, we provide an algorithm which can decompose the computation  of the worst-case SISO sensitivity of generator $i$ with respect to load $j$ into  sub-problems  involving computations on smaller graphs when $\Graph$ has bridges. 
Recall that the worst-case SISO sensitivity problem formulated in \eqref{eq:wcsen} is non convex and thereby difficult to solve for large networks. 
This algorithm aims to reduce the computational complexity by breaking the large-scale computation down into independent smaller tasks, which are usually much easier than the original problem and can be processed in parallel.

Here a bridge is an edge in $\Edge$ whose deletion disconnects the graph.
Define $\Edge^{\rm bri}$ as the set of bridges in $\Edge$.
In a not necessarily connected graph $\Graph'$, we say $v_i\Leftrightarrow v_{j+\nG}$ if there exists a path between nodes $v_i$ and $v_{j+\nG}$ .

In the first step, for any bridge $e=(u,v)\in\Edge^{\rm bri}$ that partitions $\Graph$ into $\Graph_1$ and $\Graph_2$,
 if $v_i\Leftrightarrow v_{j+\nG}$ after $e$ is deleted, then without loss of generality we assume both $v_i$ and $v_{j+\nG}$ are in $\Graph_1$.
 In this case we can replace the whole of $\Graph_2$ by a single bus. The rule is if $\Graph_2$ contains only load buses then it will be replaced by a single load, else it is replaced by a single generator.
 
 In the second step, we find a shortest path (in terms of the number of edges along the path) connecting $v_i$ and $v_{j+\nG}$,
 and the bridges along the path will partition the graph into subgraphs $\{\Graph_l(\Vertex_l,\Edge_l)\}_{l=1}^m$.
Assume the indices are assigned such that $\Graph_{l-1}$ is always closer to $v_i$ than $\Graph_l$. 
At the location of each bridge 
connecting $\Graph_{l-1}$ and $\Graph_l$, 
we add a single load $q_{l-1}$ to $\Graph_{l-1}$ and a single generator $p_{l-1}$ to $\Graph_l$, as shown in Figure \ref{Fig:Alg}.
For notational consistency,
we refer to $i$ as generator $p_0$ and $j$ as load $q_m$.
Then the computation of $C_{i\leftarrow j}^{\mathrm{wc}}$ can be composed as $\prod_{l=0}^{m-1} C_{p_l\leftarrow q_{l+1}}^{\mathrm{wc}}$,
where each $C_{p_l\leftarrow q_{l+1}}^{\mathrm{wc}}$ only depends on computing the sensitivity for smaller graphs.
This procedure is summarized as Algorithm \ref{alg2}.

\begin{figure*}[ht]
\vspace*{0.2cm}
\centering
\includegraphics[width=1.7\columnwidth]{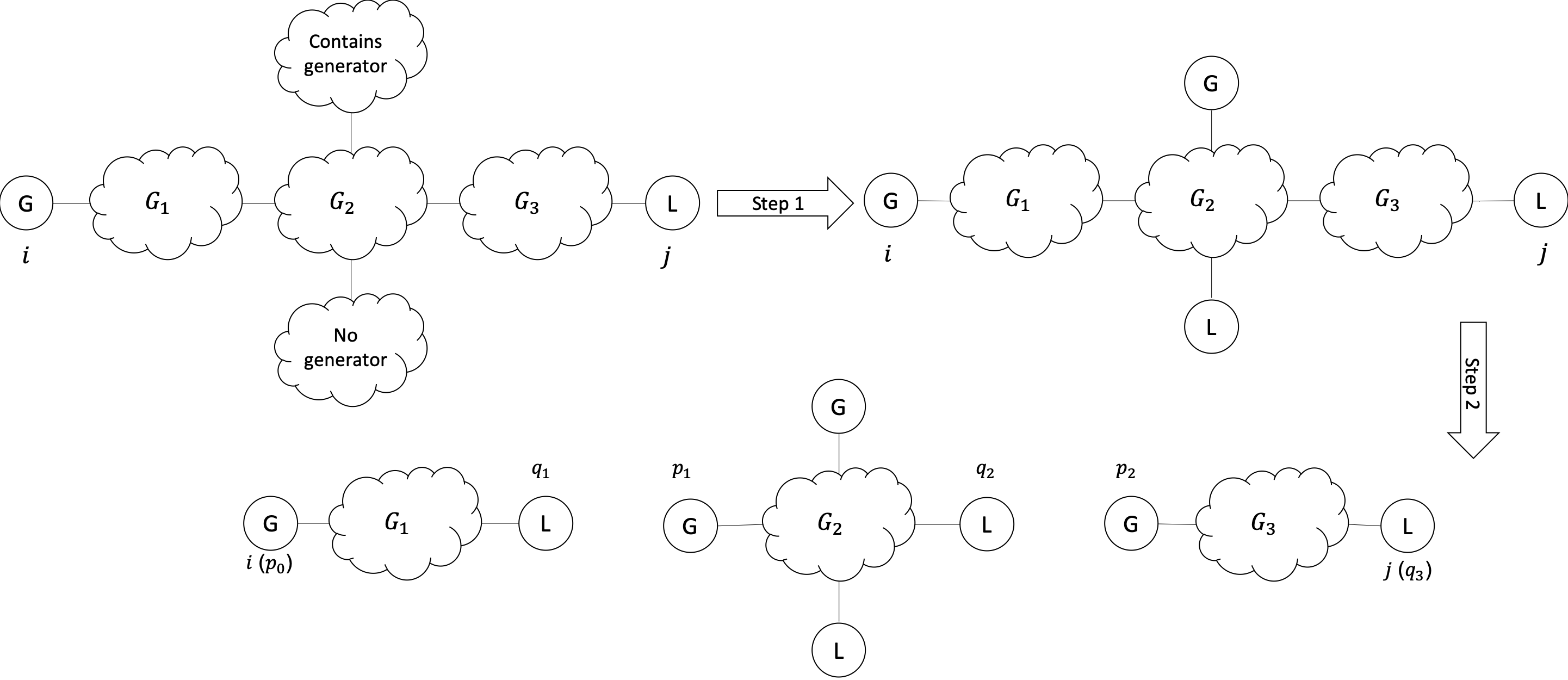}
\caption{Algorithm to decompose the worst-case SISO sensitivity of generator $i$ with respect to load $j$ into computation involving smaller graphs. Step1: Replace the off-path subgraphs by a single node. 
Step 2: Partition the on-path subgraphs and complement each subgraph by adding a pair of generator/load. }
\label{Fig:Alg}
\end{figure*}

\section{Example}\label{sec:example}
We now consider two numerical examples that demonstrate the theory and algorithm presented in the previous section. Both examples make use of the IEEE 9-bus test network, full details of the model can be found in the MATPOWER toolbox~\cite{MATPOWER}.

\subsection{9-Bus Example}
In this example the network is small enough that the decomposition algorithm of Section~\ref{sec:alg} is not necessary. The IEEE 9-bus test network is shown in Figure~\ref{Fig:chained-net}. The network consists of 3 generators, $\{v_1,v_2, v_3 \}$ and 6 loads, $\{v_4,\hdots, v_9\}$. In Table~\ref{tb:wcsen9} we have computed the worst-case SISO sensitivity for every generator load pair in the network.

\begin{figure*}
\centering
\includegraphics[height=11.2em]{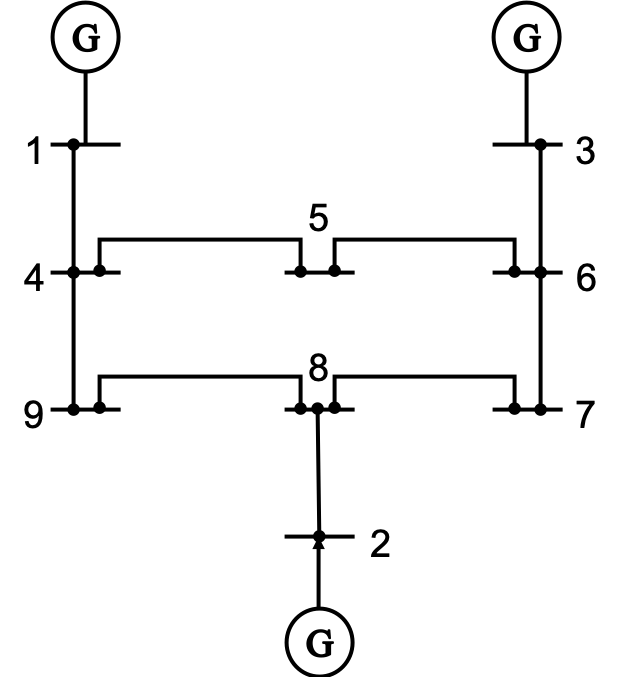}
\hspace{2em}
\includegraphics[height=12em]{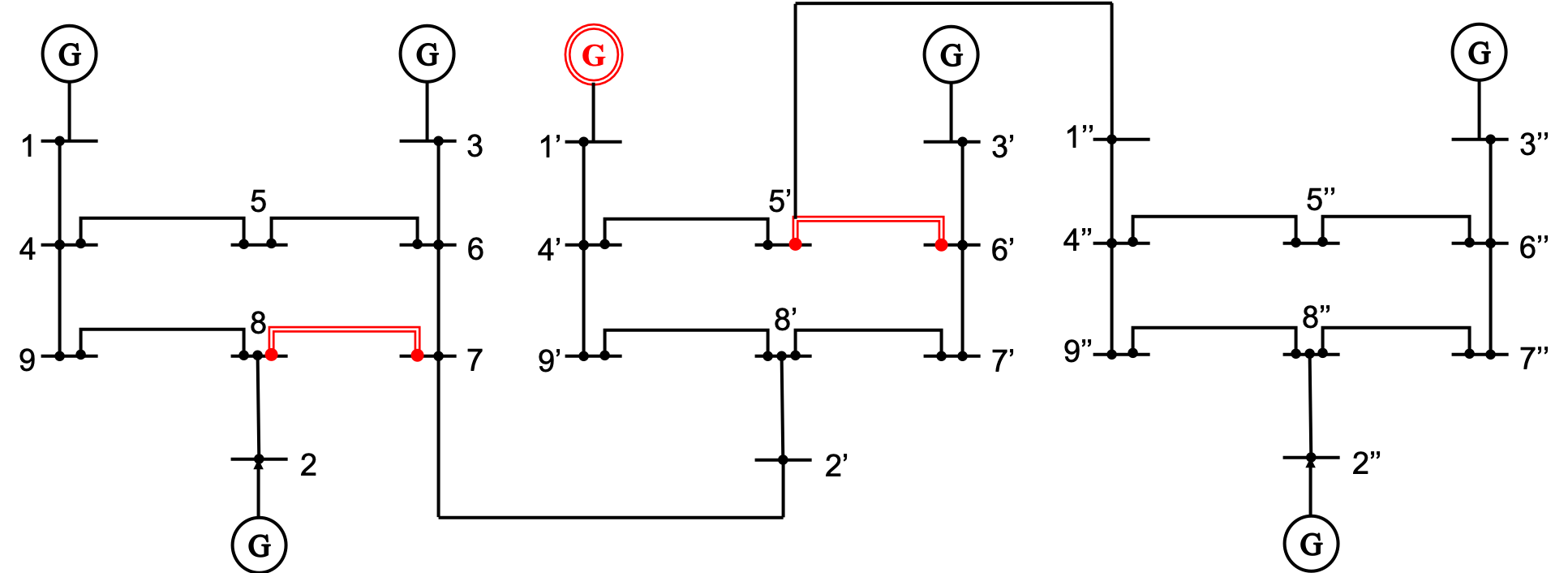}
\caption{Left: IEEE 9-bus network. Right: A 27-bus auxiliary network constructed by chaining three identical 9-bus networks together. Nodes and edges in red indicate that their corresponding generation and flow constraints are binding for every generator-load worst-case SISO sensitivity pairing in Table \ref{tb:binding}.}
\label{Fig:chained-net}
\end{figure*}

\begin{table}[h]
\vspace{0.5em}
\caption{E.g. 1: worst-case SISO sensitivity for the 9-bus network.}
\centering
\begin{tabular}{|c|c|c|c|c|c|c|}
\hline
\diagbox[height=2.5em, width=3em]{$\VertexG$}{$\VertexL$}& $4$ & $5$ & $6$ & $7$ & $8$ & $9$\\ \hline
1  & 1.0000   &  1.3935 &   2.0650 &   2.4748 &   1.9389 &   1.3244\\ \hline
2  &  2.4236  &  2.9560 &   1.7024 &   1.4748 &   1.0000 &   2.0081\\ \hline
3  &  2.5162  &  1.9838 &   1.0000 &   1.3847 &   1.6595 &   3.0081\\ \hline
\end{tabular}
\label{tb:wcsen9}
\end{table}

This example shows that the network is most sensitive to perturbations to load $v_9$ as felt by generator $v_3$. It is interesting to note that the distance (in terms of number of lines between the pair) between this pair of buses is as large as it could be for a network of this topology. The worst-case sensitivities were computed using a brute-force search over the discrete sets $(\setSgen, \setSbra)$ subject to the constraints~\eqref{eq:disc_cons}. This example is small enough for such an approach to easily be computationally tractable. In the next sub-section we consider an example where this is not the case. 

\subsection{27-Bus Example}
This example computes the worst-case SISO sensitivity of a 27-bus network. The network is constructed by chaining together three copies of the 9-bus network described in the previous example, it is illustrated   in Figure~\ref{Fig:chained-net}.  This system was chosen to demonstrate the  algorithm of Section~\ref{sec:alg} as it easily decomposes into three 9-bus subgraphs. The worst-case sensitivity can then be computed (in parallel) for each of the subgraphs, with the  global solution then given by multiplying the sensitivities of each of the subproblems together.

In Table~\ref{tb:wcsen27} we show a subset of the SISO worst-case generator-load pairs. We have chosen to show the results of the computation from loads located at the far right of the network to generators at the far left. From the decomposition algorithm, we know that these values are likely to be larger than those of pairings that are closer together because the graph in the middle, i.e. the 9-bus network with nodes labeled with a single prime, e.g. $4'$, acts as a multiplier for generator-load pairs that have a shortest path passing through it.

\begin{table}[h]
\vspace{0.5em}
\caption{E.g. 2: worst-case SISO sensitivity for the 27-bus chained network.}
\centering
\begin{tabular}{|c|c|c|c|c|c|c|}
\hline
\diagbox[height=2.5em, width=2.78em]{\!\!$\VertexG$}{$\VertexL\!\!$}& $4"$ & $5"$ & $6"$ & $7"$ & $8"$ & $9"$\\ \hline
1 &   7.3155 &   10.1942 &  15.1069 &  18.1045 &   14.1843 &    9.6889\\ \hline
2 &   4.3595 &   6.0750   &  9.0026   & 10.7889  &   8.4528   &    5.7739\\ \hline
3 &   4.0933 &   5.7040   &  8.4528   & 10.1301  &   7.9366   &    5.4213\\ \hline
\end{tabular}
\label{tb:wcsen27}
\end{table}

In Table~\ref{tb:binding} (on the next page), for every sensitivity pairing we have listed the binding constraints, i.e., the edge flows and generations that hit their limits. Observe that generator $1'$ and lines $(7, 8), (5', 6')$ are active for all pairings and hence omitted from the table (they are however marked in red in Figure~\ref{Fig:chained-net}). 

\begin{table*}
\vspace{0.5em}
\caption{Binding generators/branches corresponding to the worst-case SISO sensitivity for the 27-bus chained network.}
\centering
\begin{tabular}{|c|c|c|c|c|c|c|}
\hline
\diagbox[height=2.5em, width=3em]{$\VertexG$}{$\VertexL$}& $4"$ & $5"$ & $6"$ & $7"$ & $8"$ & $9"$\\ \hline
1 &   $3, 2"$, $(6",7")$ &   $3, 3"$, $(4", 5")$ &  $3, 3"$, $(6", 7")$ &  $3, 3"$, $(7", 8")$ &   $3, 2"$, $(6", 7")$ &    $3, 2"$, $(6", 7")$\\ \hline
2 &   $3, 2"$, $(6",7")$ &   $3, 3"$, $(4", 5")$ &  $3, 3"$, $(6", 7")$ &  $3, 3"$, $(7", 8")$ &   $3, 2"$, $(6", 7")$ &    $3, 2"$, $(6", 7")$\\ \hline
3 &   $2, 2"$, $(6",7")$ &   $2, 3"$, $(4", 5")$ &  $2, 3"$, $(6", 7")$ &  $2, 3"$, $(7", 8")$ &   $2, 2"$, $(6", 7")$ &    $2, 2"$, $(6", 7")$\\ \hline
\end{tabular}
\label{tb:binding}
\end{table*}

\section{Conclusion}
The recently developed $\OPF$ operator framework for analyzing the sensitivity of a DC optimal power flow problem was introduced and several sensitivity based analysis questions were posed. Our work highlighted the structure in the associated Jacobian matrix and provided two equivalent (under mild conditions) formulations of the Jacobian; one involving sets of binding constraints,  essentially being a ``discrete'' object composed of a graph and some sets. The second, is continuous in nature, and is constructed from optimal power flow parameters. For the worst-case SISO sensitivity problem, we proposed a decomposition algorithm (that permits parallel computation) to reduce the complexity of the combinatorial nature of the calculation and  illustrated it on a 27-bus example.

We have more structural results akin to Lemma~\ref{lm:subgraph_is_not_full} for specific network topologies that we will be publishing soon. Such results are only possible using the discrete Jacobian formulation.





\bibliographystyle{IEEEtran}
\bibliography{references}
\section{Appendix}
\subsection{OPF in Standard Form}\label{sec:standard}
In this appendix we rewrite the OPF problem~\eqref{eq:opf1} in standard form. i.e., $\{\text{minimize}_{\x}~ \mathbf c^T\x \text{ s.t. }\mathbf{A}\x \le \mathbf b\}$. Define the decision vector as 
\begin{equation*}
\x := \left[ \begin{array}{c} \sgen \\ \hdashline[0.4pt/1pt] \ang \end{array}\right],
\end{equation*}
then the inequality constraints become
\begin{align}\label{eq:standard}
\left[ \begin{array}{rr}
\mathbf 0&\base_1 \\
\mathbf 0&-\base_1 \\
 \hdashline[0.4pt/1pt]
-\mathbf W & \CL \\
\mathbf W & - \CL \\
 \hdashline[0.4pt/1pt]
\CI &\mathbf 0 \\
-\CI  &\mathbf 0 \\
 \hdashline[0.4pt/1pt]
\mathbf 0&  \B\CM^\T\\
\mathbf 0&-\B\CM^\T \\
\end{array}
\right]  \left[ \begin{array}{c} \sgen \\ \ang \end{array}\right] \le 
\left[ \begin{array}{c}
0 \\ 0 \\ \hdashline[0.4pt/1pt]\mathbf y \\ \mathbf y \\  \hdashline[0.4pt/1pt] \genulim \\ \genllim \\ \hdashline[0.4pt/1pt] \overline{\pflow} \\ \underline{\pflow}
\end{array}
\right],
\end{align}
where 
\begin{equation*}
\mathbf W := \left[ \begin{array}{c} \CI \\ \mathbf{0} \end{array}\right] \quad\text{and}\quad \mathbf{y} := 
 \left[ \begin{array}{c}\mathbf{0}  \\ \hdashline[0.4pt/1pt] -\sload \end{array}\right].
\end{equation*}
The partitions in~\eqref{eq:standard} correspond to the constraints in~\eqref{eq:opf1}. The cost vector is defined as 
\begin{equation*}
\mathbf c := \left[ \begin{array}{c} \f \\ \mathbf{0} \end{array}\right] .
\end{equation*}
Note that in this formulation, each equality constraint has been written as two inequality constraints. This has been done so as to coincide with our definition of independent binding constraints using the $\perp$ notation.

\subsection{OPF KKT Conditions}
Here we provide the KKT conditions for the optimal power flow problem~\eqref{eq:opf1}. The Lagrange multiplies are required in the definition of the set $\setf$.

Define $\taueq\in\Real^{\nGL+1}$ to be the vector of Lagrangian multipliers 
associated with equality constraints \eqref{eq:opf1.b}, \eqref{eq:opf1.c},
and $(\lambdap,\lambdam)$ and $(\mup,\mum)$ to be vectors of Lagrangian multipliers 
associated with inequalities \eqref{eq:opf1.d} and \eqref{eq:opf1.e} respectively. 
The  KKT conditions are then:
\begin{eqnarray*}
&& \eqref{eq:opf1.b}-\eqref{eq:opf1.e}\\
&& \bf{0}=M^{\T}\taueq+\CM\B(\mup-\mum)\\
&&  -\f=-[\taueq_1,\taueq_2,\cdots,\taueq_{\nG}]^{\T}+\lambdap-\lambdam\\
&& \mup,\mum,\lambdap,\lambdam\geq 0\\
&& \mup^{\T}(\B\CM^{\T}\ang-\overline{\pflow})=\mum^{\T}(\underline{\pflow}-\B\CM^{\T}\ang)=0\\
&& \lambdap^{\T}(\sgen-\genulim)=\lambdam^{\T}(\genllim-\sgen)=0,
\end{eqnarray*}
where
\begin{align*}
\M:=\left[
\begin{array}{c}
\CL\\
 \base_1^{\T}
\end{array} 
\right]
\end{align*}
is an $(\nGL+1)$-by-$\nGL$ matrix with rank $\nGL$ and $\base_1$ denotes the standard first basis vector. 

\end{document}